\documentclass{amsart}
\usepackage{graphicx} 

\usepackage{amsmath}
\usepackage{amssymb}
\usepackage{amsthm}

\usepackage{dsfont}

\usepackage{esint}

\usepackage{mathrsfs}

\theoremstyle{theorem}
\newtheorem{theorem}{Theorem}[section]

\theoremstyle{definition}

\theoremstyle{definition}

\theoremstyle{definition}

\theoremstyle{theorem}
\newtheorem{lemma}[theorem]{Lemma}

\theoremstyle{theorem}
\newtheorem{proposition}[theorem]{Proposition}

\theoremstyle{theorem}
\newtheorem{conjecture}[theorem]{Conjecture}

\theoremstyle{corollary}
\newtheorem{corollary}[theorem]{Corollary}

\theoremstyle{remark}
\newtheorem{remark}[theorem]{Remark}

\DeclareMathOperator{\Ker}{Ker}

\DeclareMathOperator{\Image}{Im}

\DeclareMathOperator{\rank}{rank}

\usepackage{polynom}

\DeclareMathOperator{\Hom}{Hom}
\DeclareMathOperator{\Sym}{Sym}

\DeclareMathOperator{\ind}{ind}

\DeclareMathOperator{\ch}{ch}
\DeclareMathOperator{\Td}{Td}

\DeclareMathOperator{\Rat}{Rat}

\DeclareMathOperator{\delbar}{\overline{\partial}}

\usepackage{comment}
\usepackage{appendix}

\usepackage{tikz-cd}
\usetikzlibrary{matrix,arrows,decorations.pathmorphing}

\title{The moduli space of multi-monopoles on a Riemann surface}
\author{Ollie Thakar}
\address{\parbox{\linewidth}{Department of Mathematics, Harvard University, Massachusetts, 02138}}
\email{othakar@math.harvard.edu}

\begin{document}

\begin{abstract}
    We study the moduli space of solutions to the Seiberg-Witten equations with $N$ spinors on a compact Riemann surface. These moduli spaces arise in a program to define a new enumerative invariant of 3-manifolds. They are also of independent interest in the geometry of algebraic curves, as they parameterize generalized divisors in Brill-Noether theory for higher rank vector bundles. We compute the Euler characteristic of these spaces, completing a computation initiated by Doan, and then compute their rational homology using spectral curves and techniques of Fulton and Lazarsfeld. 
\end{abstract}
\maketitle

\section{Introduction}

This article will study the moduli space of solutions to a simplified dimensional reduction of the multi-spinor Seiberg-Witten equations on a compact, connected Riemann surface $X$ of genus $g\geq1$. To state these equations, we introduce some notation which will be used throughout the paper. Let $h$ be a Riemannian metric on $X$ which is compatible with the complex structure, and let $\omega$ be the associated volume form. Let $E\to X$ be an $N$-dimensional hermitian vector bundle of degree $\int_Xc_1(E) =: D$ with a fixed unitary connection $B$. Let $L\to X$ be a hermitian line bundle of degree $d$ with gauge group $\mathscr{G} = \text{Map}(X, S^1)$ and complexified gauge group $\mathscr{G}^c = \text{Map}(X, \mathbb{C}^*)$. Let $\eta \in \Omega_X^2$ be a 2-form; we will use this to perturb our equations. We will require that $\int_X\eta > 2\pi d.$ The equations we will consider for a pair $(A, \alpha) \in \mathscr{A}(L)\times \Omega^0_X(L\otimes E)$ are:\begin{equation}\label{sw2}
\begin{cases}
\overline{\partial}_{A\otimes B}\alpha = 0 \\
i F_A + |\alpha|^2\omega - \eta = 0. \\
\end{cases}
\end{equation} We define the moduli space $\mathcal{N}^d_X(E)$ as the space of solutions to the above equations after taking the quotient by the gauge group $\mathscr{G}.$ The space $\mathcal{N}^d_X(E)$ is generically a complex projective manifold of dimension $\dim_\mathbb{C}\mathcal{N}^d_X(E) = Nd+D-(N-1)(g-1)$ (see Corollary \ref{manifold} and Proposition \ref{projective}.) Our main result is the computation of its rational cohomology. In Section 2, we will compute the Euler characteristic, yielding the following theorem:

\begin{theorem}\label{chi}
    The Euler characteristic of the moduli space $\mathcal{N}^d_X(E)$ is: $$\chi(\mathcal{N}^d_X(E)) = (-1)^{\dim_\mathbb{C}\mathcal{N}}N^g{{2g-2}\choose{\dim_\mathbb{C}\mathcal{N}}}.$$
\end{theorem}

\begin{remark}\label{nonempty}
    The expression for the Euler characteristic in Theorem \ref{chi} is nonzero whenever the dimension $\dim_\mathbb{C}\mathcal{N}^d_X(E) \geq 0,$ which tells us that generically, the moduli spaces $\mathcal{N}$ must be non-empty whenever they have non-negative formal dimension. This resolves the question posed in \cite[Remark 8.9]{Doan}.
\end{remark}

\begin{remark}
    In the remainder of the paper, unless otherwise stated, we will assume that the parameters $h$ and $B\in\mathscr{A}(E)$ have been chosen generically. This permits us to use transversality results proved in \cite{Doan}.  
\end{remark}

In Section 3, we will use techniques from algebraic geometry to compute the entire Poincar\'e polynomial of $\mathcal{N}^d$ and provide a path to computing the Hodge numbers. To state our result, we define the function $\Lambda:\mathbb{Z}^2_{\geq0}\to\mathbb{Z}$ by: \begin{equation}\label{Lambda}\Lambda(g, n) = \begin{cases}{{2g}\choose{n}}+{{2g}\choose{n-2}}+\dots+{{2g}\choose{0}} & \text{if}~n~\text{is~even} \\
{{2g}\choose{n}}+{{2g}\choose{n-2}}+\dots+{{2g}\choose{1}} & \text{if}~n~\text{is~odd} \\ \end{cases}.\end{equation}

\begin{theorem}\label{cohomology}
    The dimensions of the cohomology groups of $\mathcal{N}$ are given as: $$\dim H^r(\mathcal{N};\mathbb{Q}) = \begin{cases}
    \Lambda(g, r) & \text{if}~r<\dim_\mathbb{C}\mathcal{N} \\
    N^g{{2g-2}\choose{r}}-2\sum_{j=0}^r(-1)^{r-j}\Lambda(g,j) & \text{if}~r = \dim_\mathbb{C}\mathcal{N} \\
    \Lambda(g, \dim_\mathbb{R}\mathcal{N} - r) & \text{if}~\dim_\mathbb{C}\mathcal{N}<r\leq\dim_\mathbb{R}\mathcal{N} \\
    0 & \text{if}~\dim_\mathbb{R}\mathcal{N}<r \\
    \end{cases}$$
\end{theorem} This cohomology computation immediately implies some more concrete statements about the topology of the moduli spaces:

\begin{corollary}When $\dim_\mathbb{C}\mathcal{N} > 0,$ the space $\mathcal{N}^d_X(E)$ is connected.
    \begin{itemize}
        \item When $\dim_\mathbb{C}\mathcal{N} = 0,$ it consists of $N^g$ positively oriented points.
        \item When $\dim_\mathbb{C}\mathcal{N} = 1,$ it is a Riemann surface of genus $N^g(g-1)+1$.
    \end{itemize}
\end{corollary}

This corollary generalizes Doan's result that for $N =2, D = 0,$ the moduli space $\mathcal{N}$ consists of $2$ points when $g=1, d=0$ and is a Riemann surface of genus 5 when $g = 2, d = 1$ \cite[Theorem 8.1]{Doan}. Moreover, our cohomology computation agrees with the observation that $\mathcal{N}^d_X(E)$ is a projective space bundle over the Jacobian $\mathscr{J}_X$ whenever $d$ is sufficiently large. This observation was made previously by Doan in the case $N = 2, D = 0$ \cite[Theorem 8.1]{Doan}.

\begin{remark} Doan uses a different convention for $E$ and $L$ as in this paper. The symbols $L$ and $E$ in this paper would be written in Doan's paper as $L\otimes K^{1/2}_X$ and $E^*$, respectively.
\end{remark}

\subsection{Multi-monopole equations as generalized Brill-Noether theory} Of paramount importance to this paper is the fact that these moduli spaces have an interpretation purely expressible in terms of the complex geometry of the Riemann surface. 

\begin{remark} We will use script letters to denote holomorphic vector bundles, and in particular we use $\mathscr{E}$ to denote the holomorphic structure on $E$ coming from $B.$ 
\end{remark}

The following theorem of Doan is proved for $N = 2, D = 0$ but his proof easily applies to the general case:

\begin{theorem}[\cite{Doan}]\label{holomorphic}
$\mathcal{N}^d$ is homeomorphic to the space of pairs: $$\left.\left\{(A, \alpha)\in \mathscr{C} \middle\vert \begin{array}{l}
    \overline{\partial}_{A\otimes B}\alpha = 0 \\
    \alpha \not\equiv 0 \\
  \end{array} \right\}\middle/\mathscr{G}^c.\right.$$
\end{theorem}

Hence, the moduli space $\mathcal{N}$ parameterizes pairs $(\mathscr{L},\alpha)$ with $\mathscr{L}$ a holomorphic degree $d$ line bundle and $\alpha\in\mathbb{P}H^0_X(\mathscr{L}\otimes \mathscr{E})$ a holomorphic section determined up to scale. There is a map $q:\mathcal{N}^d\to\mathscr{J}_X = \mathscr{A}(L)/\mathscr{G}^c$ given by forgetting the section and identifying the Jacobian $\mathscr{J}_X$ with the component of the Picard group of degree $d.$ The fiber above $\mathscr{L} \in \mathscr{J}_X$ is the projective space $\mathbb{P}H^0_X(\mathscr{L}\otimes \mathscr{E}).$

Thus, the locus $q(\mathcal{N}^d)\subset \mathscr{J}_X$ is those line bundles $\mathscr{L}\in\text{Pic}^d_X\cong\mathscr{J}_X$ with $H^0_X(\mathscr{L}\otimes\mathscr{E})\geq 1,$ and the moduli space $\mathcal{N}^d$ is a desingularization (its \emph{canonical blow-up} in the language of \cite{ACGH}.) This locus becomes interesting when $\chi(\mathscr{L}\otimes \mathscr{E})\leq 0,$ so that $\mathscr{L}\otimes\mathscr{E}$ has no nontrivial holomorphic sections for generic $\mathscr{L}$; in this case, much like ordinary Brill-Noether theory, $q(\mathcal{N}^d)$ describes precisely the non-generic line bundles with respect to this property. When $N = 1$, this locus $q(\mathcal{N}^d)$ is precisely a translate of the Brill-Noether variety $W^0_d$ and $\mathcal{N}^d(\mathscr{E})$ parameterizes effective divisors of degree $d+D$, so it can be identified with the symmetric product $\Sym^{d+D}(X).$

\begin{remark}
    When $N = 1$, Theorem \ref{cohomology} recovers Macdonald's formula for the cohomology of symmetric products of a Riemann surface, which he proves with a different technique, namely using that $\Sym^d(X)$ is the quotient of a $\mathfrak{S}_d$-action on $\underbrace{X\times\dots\times X}_{n\text{~times}}$ \cite{Macdonald}.
\end{remark}

\begin{remark}
    There is another geometric interpretation of $\mathcal{N}^d_X(\mathscr{E})$ as a compactification of the moduli space of degree $-d$ line subbundles of $\mathscr{E}.$ For, whenever $N\geq2,$ a Zariski open subset of $\mathcal{N}^d$ consists of pairs $(\mathscr{L}, \alpha)$ with $\alpha$ nowhere vanishing. Such $\alpha$ can be thought of as an injective homeomorphism $\alpha\in H^0_X(\Hom(\mathscr{L}^{-1}, \mathscr{E}))$, defined up to scale of course. Such homomorphisms are characterized up to scale by their images, which are the degree $-d$ line subbundles of $\mathscr{E}.$ Conversely, each such line subbundle gives rise to an injective homeomorphism $\alpha\in H^0_X(\Hom(\mathscr{L}^{-1}, \mathscr{E}))$ for some $\mathscr{L}\in\text{Pic}^d_X.$
\end{remark}

\subsection{The multi-monopole equations in dimension 3}
The remainder of the introduction will not be referenced in the remainder of the paper, and is concerned with the implications of this paper for differential geometry in dimension 3. The Seiberg-Witten equations with multiple spinors are a generalization of the Seiberg-Witten equations which can be defined on a Riemannian 3-manifold $(Y, h)$ with a spin structure. Let $\mathbb{S}\to Y$ be the corresponding spinor bundle with Clifford multiplication $\rho:\bigwedge^*T^*Y\otimes\mathbb{C}\to\mathfrak{sl}(\mathbb{S})$ and Dirac operator $D:\Omega^0_Y(\mathbb{S})\to\Omega^0_Y(\mathbb{S})$. Define $\xi\to Y$ to be a hermitian line bundle and $E\to Y$ to be a hermitian vector bundle of rank $N.$ Fix a hermitian connection $B\in\mathscr{A}(E)$ and a perturbation $\eta \in \Omega^2_Y(i\mathbb{R}).$ Then, the Seiberg-Witten equations with multiple spinors are the following equations for a pair $(A, \Psi)\in \mathscr{A}(\xi)\times \Omega^0_Y(\mathbb{S}\otimes \xi\otimes E)$: $$\begin{cases}
    D_{A\otimes B}\Psi = 0, \\
    \rho(F_A) = \Psi\Psi^* - \frac12|\Psi|^2 + \eta \\
\end{cases}$$

On a 3-manifold $Y = X\times S^1,$ where $X$ is a Riemann surface, if $\xi = \pi^*(L\otimes K^{1/2}_X)$ for some line bundle $L\to X$ and $E$ is an $SU(2)$-bundle pulled back from $X$, \cite{Doan} shows that the multi-spinor Seiberg-Witten equations dimensionally reduce to the following equations for a triple: $(A, \alpha, \beta)\in \mathscr{A}(L)\times H^0(L\otimes E)\times H^0(L^*\otimes E^*\otimes K_X):$ \begin{equation}\label{sw2big}\begin{cases}
    \overline{\partial}_{AB}\alpha = \overline{\partial}_{AB}\beta = 0 \\
    \alpha\beta = 0 \\
    i*F_A+|\alpha|^2-|\beta|^2-i*\eta = 0. \\
\end{cases}\end{equation} Moreover, for generic $h, B,$ and $\eta$ on $X$ with $\int_X\eta>>0,$ Doan proves that $\beta = 0$, so the moduli space of solutions agrees precisely with $\mathcal{N}^d_X(E).$ 

\begin{remark} The sense in which our moduli spaces are \emph{simplified} from the multi-spinor Seiberg-Witten moduli spaces, as described in the introduction, is that we take $\beta = 0$ in Equation \ref{sw2big}.
\end{remark}

The Euler characteristic of this moduli space is the value of a putative 3-manifold invariant of $Y$ defined from a count of solutions to the multi-spinor Seiberg-Witten equations. These equations and putative invariants are well-studied due to their relationships with $PSL(2,\mathbb{R})$-character varieties of 3-manifolds, $\mathbb{Z}/2$-harmonic spinors, and special holonomy (see for instance \cite{Haydys, Doan-Walpuski, He-Parker, haydys-speculation, walpuski-speculation}.)

\subsection{Acknowledgements}
I owe many thanks to my advisor Professor Peter Kronheimer for his support, as well as Jiakai Li, Juan Mu\~noz-Ech\'aniz, Daniel Santiago-Alvarez, and Claudia Yao for insightful conversations. This project was conducted under the Simons Collaboration on New Structures in Low-Dimensional Topology.

\section{Characteristic Classes of the Moduli Spaces}

Recall that throughout, we will freely assume our parameters $h$ and $B\in\mathscr{A}(E)$ are chosen generically.

\subsection{Embedding $\mathcal{N}^d$ in a projective bundle}

We will define the configuration space of pairs as follows, where the sections are taken to be of regularity $L^2_k$ for $k$ large and the elements of the gauge group are taken to be in $L^2_{k+1}$: $$\mathscr{C} = \mathscr{A}(L)\times L^2_k\Omega^0_X(L\otimes E), ~~ \mathscr{B}^c = \mathscr{C}/\mathscr{G}^c.$$ 

\begin{remark}
A standard elliptic bootstrapping argument shows that the sections $\alpha$ contained in $\mathcal{N}\subset \mathscr{C}/\mathscr{G}^c$ are $C^\infty.$ 
\end{remark}

Consider the bundle $\widetilde{P}\to\mathscr{A}(L)\times X$ which is the pullback of $L\otimes E$ by projection onto the second factor. Each slice $\widetilde{P}\vert_{\{A\}\times X}$ comes with a distinguished connection $A\otimes B.$ The complex gauge group $\mathscr{G}^c$ acts on $\widetilde{P} = \mathscr{A}(L)\times \left(L\otimes E\right)$ by $g\cdot(A, x, v) = (g\cdot A, x, g\cdot v).$ Taking the quotient of $\widetilde{P}$ by this action gives us a bundle $P\to\left(\mathscr{A}(L)/\mathscr{G}^c\right)\times X\cong\mathscr{J}_X\times X$ which parameterizes a family of connections on each fiber $P\vert_{\{\mathscr{L}\}\times X}$. For each $\mathscr{L}\in\mathscr{J}_X,$ denote by $\delbar_\mathscr{L}$ the associated Dolbeault operator of the distinguished connection on $P\vert_{\{\mathscr{L}\}\times X}$. Holomorphically speaking, the bundle $P$ is the tensor product of the degree $d$ Poincar\'e line bundle with the pullback of $E\to X$ by projection onto the second factor.

For $k$ large, the Hirzebruch-Riemann-Roch Theorem computes that the index of each operator $\delbar_\mathscr{L}:L^2_k\Omega^0_X(L\otimes E)\to L^2_{k-1}\Omega^{0,1}_X(L\otimes E)$ is $$R := \ind \delbar_\mathscr{L}= N(d+1-g)+D.$$ As in the construction of the index bundle of a family in \cite{AS4}, there exists $K$ a positive integer and a linear map $\psi:\mathbb{C}^K\to L^2_{k-1}\Omega_X^{0,1}(L\otimes E)$ such that the Fredholm map $\overline{\partial}_\mathscr{L}\oplus\psi:L^2_{k}\Omega_X^{0}(L\otimes E)\oplus\mathbb{C}^K\to L^2_{k-1}\Omega_X^{0,1}(L\otimes E)$ is surjective for each $\mathscr{L}\in\mathscr{J}_X$. Moreover, we may define a bundle $p:U\to\mathscr{J}_X$ whose fiber above $\mathscr{L}\in \mathscr{J}_X$ is the space $\Ker \overline{\partial}_\mathscr{L}\oplus\psi.$ The map $(\mathscr{L}, (\alpha, v))\mapsto v$ defines a section $\tilde{s}$ of the trivial bundle $\underline{\mathbb{C}}^K\to U$ which descends to a section $s$ of the bundle $\mathscr{O}_{\mathbb{P}U}(-1)\otimes\underline{\mathbb{C}}^K$ over the fiberwise projectivization $\mathbb{P}U.$ By Theorem \ref{holomorphic}, the moduli space $\mathcal{N}$ may be identified with the zero section of $s.$

Let $\xi := c_1(\mathscr{O}_{\mathbb{P}U}(-1))\in H^2(\mathbb{P}U; \mathbb{Z}).$ Then, the Leray-Hirsch theorem and the definition of Chern classes provide a formula for the cohomology ring of $\mathbb{P}U$: \begin{equation}\label{leray-hirsch}\left. H^*(\mathbb{P}U;\mathbb{Q})\cong H^*(\mathscr{J}_X;\mathbb{Q})[\xi]\middle/\left(\sum_{i=0}^{K+R}c_i(U)\xi^{K+R-i} = 0\right)\right.\end{equation}

Let $\Theta\in H^2(\mathscr{J}_X;\mathbb{Z})$ be the cohomology class of the canonical principal polarization.

\begin{proposition}
    The section $s\in \Omega^0_{\mathbb{P}U}\left(\mathscr{O}_{\mathbb{P}U}(-1)\otimes\underline{\mathbb{C}}^K\right)$ is transverse to the zero section.
\end{proposition}

\begin{proof}
    It suffices to show that for $x=(\mathscr{L}, (\alpha, 0))\in U,$ we have $\Image d\tilde{s}_x = \underline{\mathbb{C}}^K.$ Let $v \in \mathbb{C}^K$. Then, we need to show that there exists $(b, \beta)\in T_\mathscr{L}\mathscr{J}_X\times L^2_{k}\Omega_X^{0}(L\otimes E)$ such that $\delbar_\mathscr{L}\beta + b^{0,1}\alpha = -\psi(v).$ In \cite[Theorem 7.1]{Doan}, Doan proves that the moduli space $\mathcal{N}^d$ is cut out transversely, that is that $dF$ is surjective at $x$ where $F(\mathscr{L}, \alpha) = \delbar_\mathscr{L}\alpha$. (This theorem as presently cited only applies for the case $N = 2, D = 0,$ but the general case follows easily from the same argument.) But $dF_x(b, \beta)$ is precisely $\delbar_\mathscr{L}\beta + b^{0,1}\alpha,$ which suffices for the proof.
\end{proof}

\begin{corollary} \label{manifold}
    The space $\mathcal{N}$ is a smooth, almost complex manifold of dimension $\dim_\mathbb{C}\mathcal{N} = R + g -1.$
\end{corollary}

\begin{remark}
    We will prove that $\mathcal{N}$ is in fact a complex manifold in Proposition \ref{projective}.
\end{remark}

\begin{corollary}\label{pd}
    Under the above genericity assumptions, the manifold $\mathcal{N}\subset\mathbb{P}U$ is Poincar\'e dual to $\xi^K.$
\end{corollary}
\begin{proof}
    The manifold $\mathcal{N}$ is the transverse intersection of a section of the bundle $\mathscr{O}_{\mathbb{P}U}(1)\otimes\underline{\mathbb{C}}^K$ with the zero section; hence, $$PD[\mathcal{N}] = e\left(\mathscr{O}_{\mathbb{P}U}(-1)\otimes\underline{\mathbb{C}}^K\right) = c_K\left(\mathscr{O}_{\mathbb{P}U}(-1)\otimes\underline{\mathbb{C}}^K\right)=\xi^K.$$
\end{proof}

\subsection{The tangent bundle of $\mathcal{N}$}
\begin{lemma}
    The pullback of $\mathscr{O}_{\mathbb{P}U}(1)\otimes p^*\ind\delbar_\mathscr{L}\to\mathbb{P}U$ via the inclusion $\mathcal{N}\hookrightarrow\mathbb{P}U$ is stably isomorphic to the tangent bundle $T\mathcal{N}\to \mathcal{N}.$ 
\end{lemma}
\begin{proof}
    Let $\mathscr{C}^*$ consist of all elements of $\mathscr{C}$ except those in which the section is identically zero. We define ${\mathscr{B}^c}^* = \mathscr{C}^*/\mathscr{G}^c.$ Note that projection ${\mathscr{B}^c}^*\to \mathscr{J}_X = \mathscr{A}(L)/\mathscr{G}^c$ gives ${\mathscr{B}^c}^*$ the structure of an infinite-dimensional projective space bundle over the Jacobian.

    The map $(A, \alpha) \mapsto \overline{\partial}_{AB}\alpha$ is a smooth map of Banach manifolds $F:\mathscr{B}^*\to L^2_{k-1}\Omega^{0,1}(L\otimes E)$ whose linearization is Fredholm, by ellipticity. Fix $(A, \alpha)\in{\mathscr{B}^c}^*$ and consider the linearization $dF: T_{A, \alpha}{\mathscr{B}^c}^* \to T_{F(A, \alpha)} L^2_{k-1}\Omega^{0,1}(E\otimes L).$

    Suppose we pick the background connection $B\in\mathscr{A}(E)$ generically, such that $\mathcal{N}$ is cut out transversely \cite{Doan}. Then, the tangent bundle of $\mathcal{N}$ is canonically isomorphic to the bundle composed of $\Ker dF$ at each $(A, \alpha) \in \mathcal{N},$ since $dF$ is surjective by the transversality assumption. This surjectivity also guarantees that this bundle $\Ker dF$ is equal to $\ind dF$ in the K-theory of $\mathcal{N}.$ Now, $$dF_{(A, \alpha)}(a, \beta) = \overline{\partial}_A\beta + a^{0,1}\alpha,$$ and the term $a^{0,1}\alpha$ may be homotoped out without affecting the class of $\ind dF \in K(\mathcal{N}).$ Hence, $T\mathcal{N}$ is stably isomorphic to the index bundle $\ind \delbar_\mathcal{N}$ of Dolbeault operators parameterized by $\mathcal{N}.$ 
    
    We must construct this index bundle and verify it agrees with $\mathscr{O}_{\mathbb{P}U}(1)\otimes p^*P$. The family $P\to\mathscr{J}_X\times X$ pulls back to the family $p^*P\to U\times X.$ Upon taking the projectivization of $U,$ this family descends to a family $\mathscr{O}_{\mathbb{P}U}(1)\otimes p^*P\to \mathbb{P}U\times X.$ Restricting $\mathscr{O}_{\mathbb{P}U}(1)\otimes p^*P$ to $\mathcal{N}\subset \mathbb{P}U,$ we get a family parameterizing Dolbeault operators on $L\otimes E\to X$ whose index bundle clearly recovers $\ind \delbar_\mathcal{N} \in K(\mathcal{N}).$
\end{proof}

The bundle $U\to\mathscr{J}_X$ is related to the index bundle of the family of operators $\ind\delbar_\mathscr{L}$ by: $$\ind\delbar_\mathscr{L} = [U] - [\underline{\mathbb{C}}^K]\in K(\mathscr{J}_X).$$

\begin{lemma}\label{km}
    The Chern classes of the index bundle $\ind\delbar_\mathscr{L}\to\mathscr{J}_X$ (and hence of $U$) are given as: $$c_r(U) = c_r(\ind\delbar_\mathscr{L}) = (-1)^r\frac{N^r}{r!}\Theta^r.$$
\end{lemma}
\begin{proof}
    We follow the arguments of {\cite[Proposition 9.1]{GTES2}}, which is essentially a proof of this lemma in the case $N = 2.$
    
    By the homotopy invariance of the index, we may assume that $\mathscr{E}$ is holomorphically decomposable as $\mathscr{O}_X^{\oplus (N-1)}\oplus\eta,$ where $\eta$ is a line bundle of degree $D$, so that in $K(\mathscr{J}_X)$, the index $\ind\overline{\partial}_{\mathscr{L}}$ will be $(N-1)\tau_d+\tau_{d+D},$ where we define $\tau_m$ define to be the index of the family of $\overline{\partial}$-operators on the degree $m$ Poincar\'e line bundle $\mathscr{P}_m\to X\times\mathscr{J}_X.$ Then, we apply the families Atiyah-Singer index theorem of \cite{AS4}, which tells us a formula for the Chern characters: $$\ch(\tau_m)\Td(\mathscr{J}_X) = \ch(\mathscr{P}_m)\Td(X\times\mathscr{J}_X)/[X].$$

    We first compute the Todd classes. $\mathscr{J}_X$ is an abelian variety, so its tangent bundle is trivial, thus $\Td(\mathscr{J}_X) = 1.$ We know that $\Td(X) = 1 - (g-1)\sigma,$ where $\sigma$ is the positive generator of $H^2(X;\mathbb{Z}).$

    Letting $\alpha_i, \beta_i$ be the standard generators for $H^1(X;\mathbb{Z})$ and $a_i, b_i$ be the corresponding generators of $H^1(\mathscr{J}_X;\mathbb{Z}).$ Then, considering the restrictions of $\mathscr{P}_m$ to a slice of each factor, we can compute that: $$c_1(\mathscr{P}) = m\sigma + \sum_i\left(a_i\smile\alpha_i + b_i\smile\beta_i\right).$$ The principal polarization of $\mathscr{J}_X$ satisfies $\Theta = \sum_i a_i\smile b_i$. The formula for the Chern character of a line bundle yields: $$\ch(\mathscr{P}_m) = 1 + c_1 + \frac12c_1^2+\dots = 1 + m\sigma + \sum_i\left(a_i\smile\alpha_i + b_i\smile\beta_i\right) - \Theta\smile\sigma.$$

    Now, we get that: \begin{equation*}
    \begin{split}
    \ch(\operatorname{ind}\overline{\partial}_\mathscr{L}) &= (N-1)\ch(\tau_d)+\ch(\tau_{d+D}) \\
    &=\left((N-1)\ch(\mathscr{P}_d)+\ch(\mathscr{P}_{d+D})\right)(1 - (g-1)\sigma)/[\Sigma] \\
    &=N(d - g + 1)+D - N\Theta \\
    &= R - N\Theta \\
    \end{split}
    \end{equation*}

    The Chern character vanishes in degree higher than 1, so we can compute the Chern classes inductively, which gives us: $$c_r(\operatorname{ind}\overline{\partial}_\mathscr{L}) = (-1)^r\frac {N^r}{r!}\Theta^{r},$$ as desired.
\end{proof}

For the next lemma, we must take care to define an appropriate generalization of binomial coefficients when the entries can be negative. Suppose $j\geq 0.$ We will define ${{n}\choose{j}} = 1$ if $j=0$ and: $${{n}\choose{j}} := \frac{n}j\cdot\frac{n-1}{j-1}\dots\cdot\frac{n-j+1}{1}=(-1)^j{{j-n-1}\choose{j}}$$ if $j>0.$ This way, the formal power series expansion $(1+x)^n = \sum_{j=0}^\infty{{n}\choose{j}}x^j$ holds.

\begin{lemma}\label{cherntangent}
   The Chern classes of $\mathscr{O}_{\mathbb{P}U}(1)\otimes p^*\ind\delbar_\mathscr{L}$ are given as: $$c_k(\mathscr{O}_{\mathbb{P}U}(1)\otimes p^*\ind\delbar_\mathscr{L}) = \sum_{l=0}^{\min(k, g)} {{R - l}\choose{k-l}}(-1)^l\frac{N^l}{l!}\Theta^l\xi^{k-l}$$
\end{lemma}
\begin{proof}
    In general, if $L$ is a complex line bundle over a space $Z$ and $V\to Z$ is a rank $r$ complex vector bundle, we have the following formula from the splitting principle: \begin{equation*}
        \begin{split}
            c(L\otimes V) &= \sum_{k=0}^r \sum_{l=0}^k {{r - l}\choose{k - l}}c_1(L)^{k-l}c_l(E). \\
        \end{split}
    \end{equation*} 
    
    From the description $\ind\delbar_\mathscr{L} = [U] - [\underline{\mathbb{C}}^K]\in K(\mathscr{J}_X)$ and the Whitney product formula, the Chern classes of $\mathscr{O}_{\mathbb{P}U}(1)\otimes p^*\ind\delbar_\mathscr{L}$ are uniquely determined by: $$c\left(\mathscr{O}_{\mathbb{P}U}(1)\otimes p^*\ind\delbar_\mathscr{L}\right)c\left(\mathscr{O}_{\mathbb{P}U}(1)\otimes \underline{\mathbb{C}}^K\right) = c\left(\mathscr{O}_{\mathbb{P}U}(1)\otimes p^*U\right).$$ Hence, it suffices for the proof to verify the following identity: $$\left(\sum_{k=0}^{R} \sum_{l=0}^k {{R - l}\choose{k - l}}\xi^{k-l}c_l(p^*U)\right)\left(\sum_{j=0}^{K}{{K}\choose{j}}\xi^{j}\right) = \sum_{k=0}^{R+K} \sum_{l=0}^k {{R+K - l}\choose{k -l}}\xi^{k-l}c_l(p^*U).$$ Using the identity $\sum_{j=0}^{n}{{n}\choose{j}}{{p}\choose{m-j}} = {{n+p}\choose{m}},$ which holds true for $p<0,$ we have: \begin{equation*}
        \begin{split}
            &\left(\sum_{k=0}^{R} \sum_{l=0}^k {{R - l}\choose{k - l}}\xi^{k-l}c_l(p^*U)\right)\left(\sum_{j=0}^{K}{{K}\choose{j}}\xi^{j}\right)\\ =& \sum_{k=0}^{R} \sum_{l=0}^k\sum_{j=0}^{K}{{K}\choose{j}}{{R - l}\choose{k - l}}\xi^{j+k-l}c_l(p^*U) \\
            =& \sum_{m=0}^{R+K} \sum_{l=0}^m\left(\sum_{j=0}^{K}{{K}\choose{j}}{{R - l}\choose{m-j - l}}\right)\xi^{m-l}c_l(p^*U) \\
            =& \sum_{m=0}^{R+K} \sum_{l=0}^m{{R+K - l}\choose{m -l}}\xi^{m-l}c_l(p^*U) \\
        \end{split}
    \end{equation*}
\end{proof}

\begin{lemma}
   When $r\geq0,$ we have the following formula for the pushforwards of powers of $\xi$:\begin{equation}\label{pushforward} p_*\left(\xi^{R+K-1+r}\right) = \frac{N^r}{r!}\Theta^r\end{equation}
\end{lemma}
\begin{proof}
    We proceed by induction. The case $r=0$ follows from the fact that on each fiber $\mathbb{P}(U\vert_{\mathscr{L}}),$ the form $\xi$ restricts to $-c_1(S),$ where $S$ is the universal sub-bundle. We then appeal to the well-known fact (see for instance \cite[p. 91]{ACGH}) that $\int_{\mathbb{CP}^{M}}(-c_1(S))^{M} = 1$ for all $M.$ From Equation \ref{leray-hirsch} we derive: $\xi^{R+K} = -\sum_{i=1}^{K+R}c_i(U)\xi^{K+R-i}.$ Hence we may compute our answer by induction and the push-pull formula:\begin{equation*}
        \begin{split}
            p_*\left(\xi^{R+K+r-1}\right) &= -p_*\left(\sum_{i=1}^{K+R}c_i(U)\xi^{K+R+r-1-i}\right) \\
            &= -\sum_{i=1}^{K+R}c_i(U)p_*\left(\xi^{K+R+r-1-i}\right) \\
            &= -\sum_{i=1}^{r}(-1)^i\frac{N^i}{i!}\Theta^ip_*\left(\xi^{K+R+r-1-i}\right) \\
            &= -\sum_{i=1}^{r}(-1)^i\frac{N^r}{i!(r-i)!}\Theta^r. \\
        \end{split}
    \end{equation*} The expression $-\sum_{i=1}^{r}(-1)^i\frac{N^r}{i!(r-i)!}\Theta^r$ simplifies to $\frac{N^r\Theta^r}{r!}$ upon using the identity $\sum_{i=1}^{r}(-1)^i\frac{1}{i!(r-i)!}=-1$.
\end{proof}

We are now ready to prove the formula from Theorem \ref{chi}: $$\chi(\mathcal{N}) = (-1)^{\dim_\mathbb{C}\mathcal{N}}N^g{{2g-2}\choose{\dim_\mathbb{C}\mathcal{N}}}.$$\begin{proof}[Proof of Theorem \ref{chi}]
    By Corollary \ref{manifold}, \begin{align*}\chi(\mathcal{N}) &= \int_\mathcal{N}c_{\dim_\mathbb{C}\mathcal{N}}(T\mathcal{N}) = \int_{\mathbb{P}U}PD[\mathcal{N}]c_{\dim_\mathbb{C}\mathcal{N}}(T\mathcal{N}). \\ \intertext{Using the formula for $PD[\mathcal{N}]$ from Corollary \ref{pd} and that for $c_{\dim_\mathbb{C}\mathcal{N}}(T\mathcal{N})$ from Lemma \ref{cherntangent}, we write:} \chi(\mathcal{N}) &= \sum_{l=0}^g{{R-l}\choose{R+g-1-l}}(-1)^{l}\frac{N^{l}}{l!}\int_{\mathbb{P}U}\Theta^{l}\xi^{K+R+g-1-l}. \\ \intertext{We may compute this integral over $\mathbb{P}U$ by first integrating along the fibers of $p:\mathbb{P}U\to\mathscr{J}_X$ and then integrating over $\mathscr{J}_X.$ Applying Equation \ref{pushforward} to compute the pushforwards of powers of $\xi,$ we get:}
        \chi(\mathcal{N}) &= \sum_{l=0}^g{{R-l}\choose{R+g-1-l}}(-1)^{l}\frac{N^{l}}{l!}\int_{\mathscr{J}_X}\Theta^{l}p_*\left(\xi^{K+R+g-1-l}\right)\\
        &= \sum_{l=0}^g{{R-l}\choose{R+g-1-l}}(-1)^{l}\frac{N^{g}}{l!(g-l)!}\int_{\mathscr{J}_X}\Theta^{g}.\\
    \intertext{The integral $\int_{\mathscr{J}_X}\Theta^g = g!$ and the identity $\sum_{j=0}^{n}{{n}\choose{j}}{{p}\choose{m-j}} = {{n+p}\choose{m}}$ complete the computation:}
        \chi(\mathcal{N})&= \sum_{l=0}^g(-1)^{l}{{R-l}\choose{R+g-1-l}}{{g}\choose{l}}N^g\\
        &= \sum_{l=0}^{g}(-1)^{\dim_\mathbb{C}\mathcal{N}}{{g-2}\choose{\dim_\mathbb{C}\mathcal{N} - l}}{{g}\choose{l}}N^g\\
        &= (-1)^{\dim_\mathbb{C}\mathcal{N}}N^g{{2g-2}\choose{\dim_\mathbb{C}\mathcal{N}}}.
\end{align*}\end{proof}

\section{Techniques from Algebraic Geometry}

\subsection{Spectral curves}

Following \cite{Hitchin} and particularly \cite{BNR}, we uncover much about the geometry of $\mathcal{N}^d$ by taking $\mathscr{E} = \pi_*\xi$ where $\pi:Y\to X$ is a branched cover of Riemann surfaces and $\xi\to Y$ is a holomorphic line bundle. $Y$ will be called a \emph{spectral curve} of $X.$

To define $Y,$ let $\mathscr{U}_X(N, D)$ be the moduli space of semistable vector bundles on $X$ of rank $N$ and degree $D$. We will use $g_X$ to denote the genus of $X.$ For $k\in\{1,\dots,N\},$ let $s_k \in H^0_X(K_X^{k}),$ and let $p:Q:=\mathbb{P}(\mathscr{O}\oplus K_X^{-1})\to X$ be the projection. We will generate two sections of vector bundles over $Q$. Note that $p_*(\mathscr{O}_{Q}(1))\cong\mathscr{O}\oplus K_X^{-1}$. Let $y\in H^0_Q(\mathscr{O}_Q(1))$ be the section corresponding to the constant section $1\in H^0_X(\mathscr{O}_X)$ under this isomorphism. Also, the projection formula tells us $p_*(p^*K_X\otimes \mathscr{O}_Q(1))\cong K_X\oplus\mathscr{O}_X;$ now let $x\in H^0_Q(p^*K_X\otimes\mathscr{O}_Q(1))$ be the section corresponding to the constant section $1\in H^0_X(\mathscr{O}_X)$ under this isomorphism. Let $Y$ be the zero locus of the following section: $$x^N+p^*s_1\otimes y\otimes x^{N-1}+\dots+p^*s_N\otimes y^N\in H^0_Q(p^*K_X^N\otimes\mathscr{O}(N)),$$ and let $\pi:Y\to X$ be the restriction to $Y\subset Q$ of $p:Q\to X.$ It is not difficult to show that $Y$ is smooth and connected for generic choices of the sections $s_k$ and that in this case, $\pi:Y\to X$ is a finite map which is generically $N$-sheeted.

\begin{theorem}[{\cite[Theorem 1]{BNR}}]
For any integer $D,$ the rational map $\pi_*:\operatorname{Pic}^\delta(Y)\dashrightarrow \mathscr{U}_X(N, D)$, which takes a line bundle $\xi\in\operatorname{Pic}^\delta(Y)$ to the direct image $\pi_*\xi$ of its sheaf of sections, is dominant and defined outside a subset of codimension $\geq 2$, where $\delta = D - \deg\pi_*(\mathscr{O}_Y).$
\end{theorem}

From the construction of $Y$, it follows that: $$\pi_*\mathscr{O}_Y \cong \mathscr{O}_X\oplus K_X^{-1}\oplus\dots\oplus K_X^{-N+1},$$ so we may compute $\delta = D + N(N-1)(g_X-1)$ and $g_Y = N^2(g_X-1)+1$. 

Let $S\subset \mathscr{U}_X(N, D)$ be the image of $\pi_*$. If $\mathscr{E} = \pi_*\xi$ is in $S,$ then so is $\mathscr{E}\otimes\mathscr{L}$ for all $\mathscr{L}\in\operatorname{Pic}^0(X)$, since by the projection formula $\mathscr{E}\otimes\mathscr{L} = \pi_*\left(\xi\otimes\pi^*\mathscr{L}\right).$ Therefore, the restriction of $S$ to the subspace $\mathscr{SU}_X(n, \Lambda)\subset\mathscr{U}_X(n, d)$ consisting of vector bundles with a fixed determinant $\Lambda$ is also Zariski open. This tells us that if we choose a holomorphic connection $B\in\mathscr{A}(E)$ generically, giving $E$ a holomorphic structure $\mathscr{E},$ i.e. so that $\mathscr{E}$ lives in the image of the above rational map, there exists a line bundle $\xi\to Y$ such that $\pi_*\xi = \mathscr{E}.$

For $\mathscr{L}\to X$ a holomorphic line bundle, the projection formula gives us an isomorphism $T:H^0_X(\mathscr{E}\otimes \mathscr{L}) \xrightarrow\cong H^0_Y(\xi\otimes\pi^*\mathscr{L}).$ Recall that the moduli space $\mathcal{N}^d_X(\mathscr{E})$ is the space of pairs $(\mathscr{L},\alpha)$ where $\mathscr{L}\in\text{Pic}^d_X$ and $\alpha\in\mathbb{P}H^0_X(\mathscr{E}\otimes \mathscr{L}).$ Clearly this is homeomorphic to the space of pairs $(\mathscr{L}, \beta)$ where $L\in\text{Pic}^d_X$ and $\beta\in\mathbb{P}H^0_Y(\xi\otimes \pi^*\mathscr{L}).$ But, the space $\mathbb{P}H^0_Y(\xi\otimes \pi^*\mathscr{L})$ is also isomorphic to the space of effective divisors on $Y$ in the linear equivalence class of $\xi\otimes \pi^*\mathscr{L}.$

The degree of $\xi\otimes \pi^*\mathscr{L}$ in $Y$ is $\deg \xi\otimes \pi^*L = \delta+Nd.$ So, the following is a fiber square, where $\mu_Y$ is the Abel-Jacobi map for $Y.$

\[\begin{tikzcd}
	{\mathcal{N}^d_X(\mathscr{E})} &&&& {\text{Sym}^{\delta+Nd}(Y)} \\
	\\
	{\text{Pic}_X^d} &&&& {\text{Pic}_Y^{\delta+Nd}}
	\arrow["{[\mathscr{L}, \alpha]\mapsto\text{Div}(T(\alpha))}", from=1-1, to=1-5]
	\arrow["{[\mathscr{L}, \alpha]\mapsto \mathscr{L}}"', from=1-1, to=3-1]
	\arrow["{\mu_Y}"', from=1-5, to=3-5]
	\arrow["{\mathscr{L}\mapsto \xi\otimes\pi^*\mathscr{L}}", from=3-1, to=3-5]
\end{tikzcd}\]

Since the map $\mathscr{L}\mapsto \xi\otimes\pi^*\mathscr{L}$ is injective, $\mathcal{N}^d_X(\mathscr{E})$ is embedded in $\Sym^{\delta+Nd}(Y)$ as $\mu_Y^{-1}(\xi\otimes\pi_*\text{Pic}^d_X).$ This tells us:

\begin{proposition}\label{projective}
For generic $\mathscr{E},$ $\mathcal{N}^d_X(\mathscr{E})$ is a complex projective manifold.
\end{proposition}

\subsection{Degeneracy locus construction}

In the previous section, we used techniques from index theory to embed $\mathcal{N}$ in a finite-dimensional projective space bundle over the Jacobian $\mathscr{J}_X.$ We will now produce a potentially different embedding of $\mathcal{N}$ in a finite-dimensional projective space bundle over the Jacobian, using algebraic geometry. This new perspective will prove useful in computing the cohomology of $\mathcal{N}.$ 

Specifically, we can exhibit $q(\mathcal{N}^d)$ as a degeneracy locus of a map of vector bundles over the Jacobian $\mathscr{J}_X$, using a strategy described in \cite{Kleiman-Laksov} and \cite{Narasimhan-Ramanan_2theta}. To fix notation, define the \emph{degeneracy locus} of a map of vector bundles $\sigma:U\to V$ over a manifold $M$ to be: $$D_k(\sigma) := \{x\in M | \rank(\sigma_x) \leq k\}.$$ Let $Z$ be the divisor $x_1+\dots+x_K$ on $\Sigma,$ for $K$ large, and consider the short exact sequence of sheaves: $$0\to\mathscr{O}(-Z)\to\mathscr{O}_X\to\mathscr{O}_Z\to0.$$ For any holomorphic line bundle $\mathscr{L}\to X,$ we may tensor this sequence with $\mathscr{L}\otimes\mathscr{E}.$

Let $K$ be so large that $H^0(\mathscr{L}\otimes\mathscr{E}\otimes\mathscr{O}(-Z)) = 0.$ Then, the long exact sequence associated to the above sheaf sequence tells us that $H^0_\Sigma(\mathscr{L}\otimes\mathscr{E})$ is isomorphic to the kernel of the connecting homomorphism: $$\sigma_\mathscr{L}:\oplus_{i=1}^K(\mathscr{L}\otimes\mathscr{E})_{x_i}\to H^1(\mathscr{L}\otimes\mathscr{E}\otimes \mathscr{O}(-Z)).$$ As $L$ varies over $\text{Pic}_X^{d},$ each of these quantities above becomes the fiber of a vector bundle over $\text{Pic}_X^{d}$. To construct these vector bundles, let $\mathscr{P}_d\to \text{Pic}^d_X\times X$ be the Poincar\'e bundle parameterizing degree $d$ line bundles and $p:\text{Pic}^d_X\times X\to\text{Pic}^d_X$ the projection map. Define the following sheaves over $\text{Pic}^d_X$: $$V_X(\mathscr{E}) = p_*(\mathscr{E}\otimes\mathscr{P}_d\otimes\mathscr{O}_Z), ~~W_X(\mathscr{E}) = R^1p_*(\mathscr{E}\otimes\mathscr{P}_d\otimes\mathscr{O}(-Z)).$$ Then, for $K$ sufficiently large, $\pi_V:V_X(\mathscr{E})\to\text{Pic}_X^{d}$ is the vector bundle whose fiber over $\mathscr{L}$ is $\oplus_{i=1}^K(\mathscr{L}\otimes\mathscr{E})_{x_i}$ and $\pi_W:W_X(\mathscr{E})\to\text{Pic}_X^{d}$ is the vector bundle whose fiber over $\mathscr{L}$ is $H^1(\mathscr{L}\otimes\mathscr{E}\otimes \mathscr{O}(-Z)).$ 

These maps $\sigma_\mathscr{L}:V_X(\mathscr{E})_\mathscr{L}\to W_X(\mathscr{E})_\mathscr{L}$ glue together to a vector bundle homomorphism $\sigma:V_X(\mathscr{E})\to W_X(\mathscr{E})$. The image $q(\mathcal{N}^d)$ is the degeneracy locus $D_{NK-1}(\sigma)$ of this map. 

\subsection{Rational cohomology of $\mathcal{N}$}
This section is devoted to computing the Poincar\'e polynomials of the moduli spaces $\mathcal{N}^d_X(\mathscr{E}),$ making use of the degeneracy locus construction and spectral curves. Following \cite{FL}, we consider the total space of the projective bundle $\pi_V:\mathbb{P}V_X(\mathscr{E})\to\text{Pic}^d_X$. The preimage $\pi_V^{-1}(D_{KN-1}(\sigma))\subset \mathbb{P}V_X(\mathscr{E})$, which consists of pairs $(\mathscr{L}, \alpha\in \mathbb{P}H^0_X(\mathscr{L}\otimes\mathscr{E}))$, is homeomorphic to $\mathcal{N}.$ We will need to recall Hartshorne's notion of ampleness for vector bundles \cite{Hartshorne}: a holomorphic vector bundle $\eta\to M$ over a variety $M$ is \emph{ample} if the tautological line bundle $\mathscr{O}_{\mathbb{P}\eta^*}(1)\to\mathbb{P}\eta^*$ is ample over the total space of the projectivized dual bundle. The following theorem of Fulton and Lazarsfeld will be our main tool:

\begin{theorem}[{\cite[Theorem II and Remark 1.9]{FL}}]\label{connectedness}
Let $\mathscr{V}, \mathscr{W}$ be vector bundles of dimensions $v, w$ respectively over a smooth irreducible complex projective variety $M,$ with $\sigma:\mathscr{V}\to\mathscr{W}$ a map of vector bundles. Let $\pi_\mathscr{V}:\mathbb{P}\mathscr{V}\to M$ be the projection. If $\Hom(\mathscr{V}, \mathscr{W})$ is ample, then the relative homology groups $H_i(\mathbb{P}\mathscr{V}, \pi_\mathscr{V}^{-1}(D_{k}(\sigma)); \mathbb{Z})$ vanish whenever $i \leq \dim_\mathbb{C}M - (v-k)(w-k).$
\end{theorem}

In order to use this theorem, we must verify that $\Hom(V_X(\mathscr{E}), W_X(\mathscr{E}))\to\text{Pic}^d_X$ is ample, which will require spectral curves.

\begin{lemma}\label{ample}
    For generic $\mathscr{E},$ the bundle $\Hom(V_X(\mathscr{E}), W_X(\mathscr{E}))\to\operatorname{Pic}^d_X$ is ample.
\end{lemma}
\begin{proof}
For each $m\in\mathbb{Z}$ and line bundle $\xi\to Y$ we can form the vector bundle $W_Y(\xi)\to\text{Pic}^{m}_Y$ as in the degeneracy locus construction above, where instead of the divisor $Z\subset X$ we use $\pi^*Z\subset Y.$ For a curve $C,$ let $\mathscr{P}_m(C)\to \text{Pic}^m_C\times C$ denote the Poincar\'e bundle parameterizing degree $m$ line bundles on $C.$

The bundle $V_X(\mathscr{E})\to\text{Pic}^d_X$ is algebraically equivalent to the trivial bundle, so it suffices to show that the bundle $W_X(\mathscr{E})$ is ample (see \cite[pp. 280-281]{FL}.) By the results of Section 3.1, a generic choice of $\mathscr{E}$ can be written as $\mathscr{E} = \pi_*\xi$ for some holomorphic line bundle $\xi\to Y$. I claim that the bundle $W_X(\mathscr{E})$ is isomorphic to the restriction to a copy $\text{Pic}_X^d\subset \text{Pic}_Y^{\delta+Nd}$ of $W_Y(\xi)$. To verify this claim, consider the following diagram (where all the vertical arrows are projections onto one of the factors):
\[\begin{tikzcd}
	Y && Y && X \\
	\\
	{\text{Pic}^{Nd}_Y\times Y} && {\text{Pic}^d_X\times Y} && {\text{Pic}^d_X\times X} \\
	\\
	{\text{Pic}^{Nd}_Y} && {\text{Pic}^d_X} && {\text{Pic}^d_X}
	\arrow["{\text{id}}"', from=1-3, to=1-1]
	\arrow["\pi", from=1-3, to=1-5]
	\arrow["r"', from=3-1, to=1-1]
	\arrow["p", from=3-1, to=5-1]
	\arrow["{r_Y}"', from=3-3, to=1-3]
	\arrow["{\pi^*\times\text{id}}"', from=3-3, to=3-1]
	\arrow["{\text{id}\times\pi}", from=3-3, to=3-5]
	\arrow["{p_Y}", from=3-3, to=5-3]
	\arrow["{r_X}"', from=3-5, to=1-5]
	\arrow["{p_X}", from=3-5, to=5-5]
	\arrow["{\pi^*}"', from=5-3, to=5-1]
	\arrow["{\text{id}}", from=5-3, to=5-5]
\end{tikzcd}\]

Since $\pi$ and the vertical maps are flat, we have the following base change isomorphism: $r_X^*\mathscr{E} =r_X^*\pi_*\xi ={(\text{id}\times\pi)}_*r_Y^*\xi.$ Using this and the projection formula, we have: 
    \begin{align*}
        W_X(\mathscr{E}) &= {R^1p_X}_*\left(r_X^*\mathscr{E}\otimes r_X^*\mathscr{O}(-Z)\otimes \mathscr{P}_d(X)\right) \\
        &= {R^1p_X}_*({(\text{id}\times\pi)}_*r_Y^*\xi\otimes r_X^*\mathscr{O}(-Z)\otimes \mathscr{P}_d(X)) \\
        &= {R^1p_X}_*{(\text{id}\times\pi)}_*(r_Y^*\xi\otimes r_Y^*\mathscr{O}(-\pi^*Z)\otimes {(\text{id}\times\pi)}^*\mathscr{P}_d(X)). \\
    \intertext{The finiteness of the map $\text{id}\times\pi:\text{Pic}^d_X\times Y\to\text{Pic}^d_X\times X$ guarantees that $${R^1p_X}_*\circ{(\text{id}\times\pi)}_* = R^1(p_X\circ(\text{id}\times\pi))_* = R^1{p_Y}_*,$$ giving us:}
        W_X(\mathscr{E}) &= {R^1p_Y}_*(r_Y^*\xi\otimes r_Y^*\mathscr{O}(-\pi^*Z)\otimes {(\text{id}\times\pi)}^*\mathscr{P}_d(X)). \\
    \intertext{The isomorphism $(\pi^*\times\text{id})^*\mathscr{P}_{Nd}(Y)\cong (\text{id}\times\pi)^*\mathscr{P}_d(X)$ of vector bundles on $\text{Pic}^d_X\times Y$ and the projection formula let us compute:}
    W_X(\mathscr{E}) &= {R^1p_Y}_*(r_Y^*\xi\otimes r_Y^*\mathscr{O}(-\pi^*Z)\otimes (\pi^*\times\text{id})^*\mathscr{P}_{Nd}(Y)) \\
        &= {R^1p_Y}_*(\pi^*\times\text{id})^*(r^*\xi\otimes r^*\mathscr{O}(-\pi^*Z)\otimes \mathscr{P}_{Nd}(Y)). \\
    \intertext{Finally, the base change isomorphism ${R^1p_Y}_*(\pi^*\times\text{id})^* = \pi^*{R^1p}_*$ gives:}
    W_X(\mathscr{E}) &= \pi^*{R^1p}_*(r^*\xi\otimes r^*\mathscr{O}(-\pi^*Z)\otimes \mathscr{P}_{Nd}(Y)) \\
        &= \pi^*W_Y(\xi). \\
    \end{align*} For sufficiently large $K,$ the bundle $W_Y(\xi)\to\text{Pic}^{\delta+ND}_Y$ is ample by \cite[Lemma 2.2]{FL}. Since the restriction of an ample vector bundle to a closed subvariety is ample \cite[Proposition 4.1]{Hartshorne}, $W_X(\mathscr{E})$ is ample, as desired.
\end{proof}

\begin{proof}[Proof of Theorem \ref{cohomology}]
By Lemma \ref{ample}, Theorem \ref{connectedness} applies to our degeneracy locus construction, so the relative homology groups $H_i(\mathbb{P}V_X(\mathscr{E}), \mathcal{N}; \mathbb{Z})$ vanish whenever $i \leq \dim_\mathbb{C}\mathcal{N}.$ The long exact sequence of a pair therefore tells us that: \begin{equation}\label{les}H_i(\mathcal{N};\mathbb{Z})\cong H_i(\mathbb{P}V_X(\mathscr{E});\mathbb{Z})~~\text{when}~~i < \dim_\mathbb{C}\mathcal{N}.\end{equation} This fact plus Poincar\'e duality and our computation of $\chi(\mathcal{N})$ determines the rational cohomology of $\mathcal{N}.$ Indeed, the Poincar\'e polynomial of $\mathbb{P}V_X(\mathscr{E})$ may be computed using the Leray-Hirsch theorem: $$P_t(\mathbb{P}V_X(\mathscr{E})) = (1+t^2+\dots+t^{2(R+K-1)})(1+t)^{2g} = \sum_{i=0}^{2g}\sum_{j=0}^{R+K-1}{{2g}\choose{i}}t^{i+2j}.$$ Then, using the definition of the function $\Lambda$ from Equation \ref{Lambda}, we recover the formula from Theorem \ref{cohomology}: $$\dim H^r(\mathcal{N};\mathbb{Q}) = \begin{cases}
    \Lambda(g, r) & \text{if}~r<\dim_\mathbb{C}\mathcal{N} \\
    N^g{{2g-2}\choose{r}}-2\sum_{j=0}^r(-1)^{r-j}\Lambda(g,j) & \text{if}~r = \dim_\mathbb{C}\mathcal{N} \\
    \Lambda(g, \dim_\mathbb{R}\mathcal{N} - r) & \text{if}~\dim_\mathbb{C}\mathcal{N}<r\leq\dim_\mathbb{R}\mathcal{N} \\
    0 & \text{if}~\dim_\mathbb{R}\mathcal{N}<r \\
    \end{cases}$$
\end{proof}

\begin{remark}
    In principle the Hodge numbers of $\mathcal{N}$ are also computable. The functoriality of the Hodge decomposition combined with Equation \ref{les} show that $h^{p,q}(\mathcal{N}) = h^{p,q}(\mathbb{P}V_X(\mathscr{E}))$ whenever $p+q<\dim_\mathbb{C}\mathcal{N}.$ Since $V_X(\mathscr{E})$ is algebraically equivalent to the trivial bundle (see \cite{FL}), these Hodge numbers can be computed using the K\"unneth theorem for Dolbeault cohomology as in \cite{GH}: $$h^{p,q}(\mathbb{P}V_X(\mathscr{E})) = h^{p,q}\left(\mathbb{P}^{NK-1}\times\mathbb{C}^g/(\mathbb{Z}+i\mathbb{Z})^g\right) = \sum_{r}{{g}\choose{p-r}}{{g}\choose{q-r}}.$$ Thus, the remaining Hodge numbers, namely those with $p+q = \dim_\mathbb{C}\mathcal{N},$ are determined by the holomorphic Euler characteristics $\chi(\Omega^pT\mathcal{N})$ for $p\in\{0,1,\dots,\dim_\mathbb{C}\mathcal{N}\}.$ The Hirzebruch-Riemann-Roch theorem tells us we may compute these as: $$\chi(\Omega^p\mathcal{N}) = \int_\mathcal{N}\ch(\Omega^p\mathcal{N})\Td(T\mathcal{N}) = \int_{\mathbb{P}U}\xi^K\ch(\Omega^p\mathcal{N})\Td(T\mathcal{N}).$$ Although the Chern characters and Todd class above are determined algorithmically from the characteristic class computations of Section 2, a closed-form expression for these cohomology classes would likely be exceedingly complicated.
\end{remark}

\bibliographystyle{alpha}
\bibliography{second_draft}
\end{document}